\documentclass[11pt]{amsart}

\usepackage[left=1in,top=1in,right=1in]{geometry}

\usepackage{
color,
amsmath,
amsfonts,
latexsym, 
amsrefs,
amssymb,
amsthm,
hyperref,
ulem,
bbm,
enumerate,
graphicx
}

\newcommand{\ii}{\mathrm{i}} 
\newcommand{\N}{\mathbb N}
\newcommand{\R}{\mathbb R}
\newcommand{\C}{\mathbb C}
\newcommand{\1}{\mathbbm{1}}
\newcommand{\rd}{{\rm d}}
\renewcommand{\P}{\mathbb{P}}

\newcommand{\E}{\mathbb{E}}
\renewcommand{\Re}{{\rm Re}}
\newcommand{\e}{\varepsilon}

\DeclareMathOperator{\OO}{O}
\DeclareMathOperator{\oo}{o}

\newtheorem{thm}{Theorem}
\newtheorem{lem}[thm]{Lemma}
\newtheorem{prop}[thm]{Proposition}

\newtheorem{rem}[thm]{Remark}
\newtheorem{conj}[thm]{Conjecture}

\title[Is the Riemann zeta function in a short interval a 1-RSB spin glass ?]{Is the Riemann zeta function in a short interval \\ a 1-RSB spin glass ?}

\author[L.-P. Arguin]{Louis-Pierre Arguin}
\address{Department of Mathematics, Baruch College and Graduate Center, City University of New York, New York, NY 10010.}
\email{louis-pierre.arguin@baruch.cuny.edu}

\author[W. Tai]{Warren Tai}
\address{Department of Mathematics, Graduate Center, City University of New York, New York, NY 10010.}
\email{wtai@gradcenter.cuny.edu}

 \subjclass[2000]{}

\date{September 13, 2018}

\begin{document}

\maketitle

\begin{abstract}
Fyodorov, Hiary \& Keating established an intriguing connection between the maxima of log-correlated processes and the ones of the Riemann zeta function on a short interval of the critical line.
In particular, they suggest that the analogue of the free energy of the Riemann zeta function is identical to the one of the Random Energy Model in spin glasses.
In this paper, the connection between spin glasses and the Riemann zeta function is explored further. We study a random model of the Riemann zeta function and show that its two-overlap distribution corresponds to the one of a one-step replica symmetry breaking (1-RSB) spin glass. 
This provides evidence that the local maxima of the zeta function are strongly clustered. 
\end{abstract}

\section{Introduction and Main Result}
\subsection{Background}
The Riemann zeta function is defined on $\C$ by
\begin{equation}
\label{eqn: zeta}
\zeta(s)= \sum_{n\geq 1}\frac{1}{n^s}=\prod_{p\text{ primes}}\left(1-p^{-s}\right)^{-1} \text{ if $\Re \ s>1$,} 
\end{equation}
and can be analytically continued to the whole complex plane by the functional equation
$$
\zeta(s)=\chi(s) \zeta(1-s)\ , \qquad \chi(s)=2^s \pi^{s-1}\sin \left(\frac{\pi}{2}s\right)\Gamma(1-s)\ .
$$
Trivial zeros are located at negative even integers where $\chi(s)=0$. The non-trivial zeros are restricted to the critical strip $0\leq \Re \ s\leq 1$. 
The Riemann hypothesis states that they all lie on the critical line $\Re\ s=1/2$. A weaker statement, yet with deep implications on the distribution of the primes, is the Lindel\"of hypothesis which stipulates that the maximum of $\zeta$ on a large interval $[0,T]$ of the critical line grows slower than any power of $T$, i.e.~$\zeta(1/2+\ii T)$ is $\OO(T^\e)$ for any $\e>0$, see e.g.~\cite{titchmarsh}.
 
Mathematical physics has provided several important insights in the study of the Riemann zeta function over the years.
We refer the reader to \cite{schumayer-hutchinson} for a broad discussion on this topic. 
We briefly highlight three contributions from statistical mechanics and probability. 
First, there are deep connections between the statistics of eigenvalues of random matrices and the zeros of zeta as exemplified by the Montgomery's pair correlation conjecture, see for example \cite{bourgade-keating}. Second, the Riemann hypothesis can be recast in the framework of Ising models of statistical mechanics where it bears a resemblance to the Lee-Yang theorem. This perspective was investigated in details by Newman \cite{newman1,newman2,newman3}. It led to an equivalent reformulation of the Riemann hypothesis in terms of the exact value of the de Bruijn--Newman constant \cite{newman}, see \cite{saouter_etal} for a numerical estimate of the constant and \cite{tao} for a proof that the constant is non-negative.
Third, Fyodorov, Hiary \& Keating \cite{fyodorov-hiary-keating} and Fyodorov \& Keating \cite{fyodorov-keating} recently unraveled a striking connection between the local statistics of the large values of the Riemann zeta function on the critical line and the extremes of a class of disordered systems, the {\it log-correlated processes}, that includes among others branching Brownian motion and the two-dimensional Gaussian free field. This connection has also been extended recently to the theory of Gaussian multiplicative chaos by Saksman \& Webb \cite{saksman-webb_1,saksman-webb_2}. 

The Fyodorov-Hiary-Keating conjecture is as follows \cite{fyodorov-hiary-keating, fyodorov-keating}: if $\tau$ is sampled uniformly on a large interval $[T,2T]$, then the maximum on a short interval, say $[0,1]$, around $\tau$ is
\begin{equation}
\label{eqn: FHK}
\max_{h\in[0,1]}\log \left| \zeta(1/2+\ii (\tau+h)\right|=\log\log T -\frac{3}{4}\log\log\log T+\mathcal M_{T} ,
\end{equation}
where $(\mathcal M_T)$ is a sequence of random variables converging in distribution. 
The deterministic order of the maximum corresponds exactly to the one of a log-correlated process, such as a branching random walk and the two-dimensional Gaussian free field, see for example \cite{kistler, arguin} for more background on this class of processes. The precise value of the leading order can be predicted heuristically since the process for $\log\zeta$ has effectively $\log T$ distinct values on $[0,1]$ (because there are on average $\log T$ zeros on $[0,1]$, see for example \cite{titchmarsh}), and the marginal distribution of $\log |\zeta(1/2+i(\tau+h))|$ should be close to Gaussian with variance $\frac{1}{2}\log\log T$ as predicted by Selberg's Central Limit Theorem \cite{radziwill-soundararajan}. The log-correlations already appear at the level of the typical values from the multivariate CLT proved in \cite{bourgade}. The first order of the conjecture \eqref{eqn: FHK} was proved recently in parallel: conditionally on the Riemann hypothesis in \cite{najnudel}, and unconditionally in \cite{ABBRS}. 
The evidence in favor of the conjecture laid out by Fyodorov \& Keating \cite{fyodorov-keating} suggests that the large values of the Riemann zeta function locally behaves like a disordered system of the spin-glass type characterized by an energy landscape with multiple minima, see Figure \ref{fig: landscape}.
\begin{figure}
\label{fig: landscape}
\includegraphics[height=5cm]{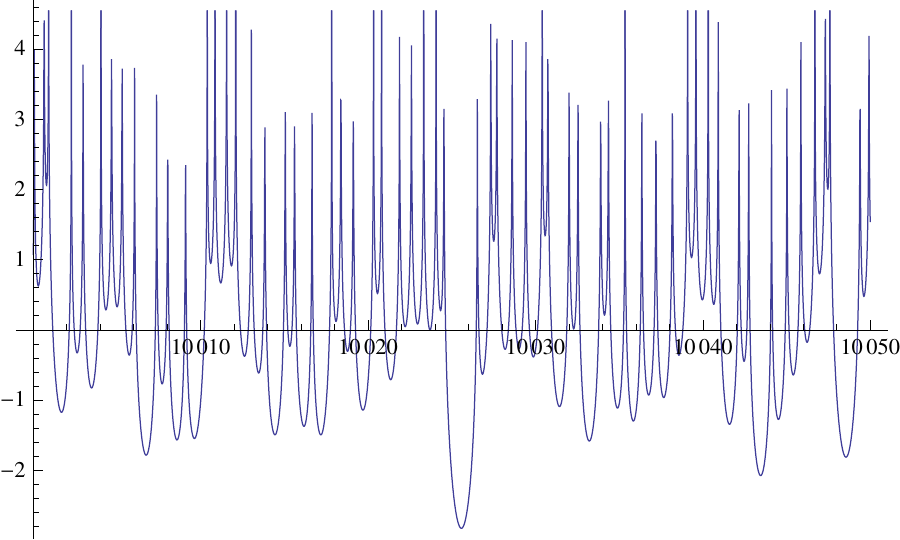}
\caption{The value of $-\log|\zeta(1/2+i(T+h))|$ for $T=10000$ and $h\in[0,50]$.}
\end{figure}
In particular, by considering $-\log |\zeta(1/2+\ii(\tau+h))|$ as the energy of a disordered system on the state space $[0,1]$,
they predict that the analogue of the free energy is in the limit
\begin{equation}
\label{eqn: free zeta}
\begin{aligned}
\lim_{T\to\infty} \frac{1}{\log \log T} \log \left(\log T \cdot \int_0^{1} |\zeta(\frac{1}{2}+\rm i(\tau+h))|^{\beta} {\rm d}h\right)=
\begin{cases}
1+\frac{\beta^2}{4} & \ \text{if $\beta<2$,}\\
\beta  & \ \text{if $\beta\geq 2$,}
\end{cases}
\end{aligned}
\end{equation}
similarly to a Random Energy Model (REM) with $\log T$ independent Gaussian variables of variance $\frac{1}{2}\log\log T$. 

In this paper, we explore the connection with spin glasses further by providing evidence that $\log|\zeta|$ behaves locally like a spin glass with one-step replica symmetry breaking (1-RSB), cf.~Theorem \ref{thm: main}. 
More precisely, we study a simple random model introduced by Harper \cite{harper} for the large values of $\log |\zeta|$. We show that two points sampled from the Gibbs measure at low temperature have correlation coefficients (or {\it overlap}) $0$ or $1$ in the limit, similarly to a 1-RSB spin glass.
We expect that part of our approach could be extended to prove a similar result for the Riemann zeta function itself as stated in Conjecture \ref{conj: zeta} below.

\subsection{The model and main result}
Let $(U_p, p \text{ primes})$ be IID uniform random variables on the unit circle in $\C$. We write $\E$ for the expectation over the $U_p$'s. 
We study the stochastic process
\begin{equation}
\label{eqn: X}
X_h=\sum_{ p\leq T} \frac{\Re(U_p p^{-ih})}{p^{1/2}}, \qquad h\in [0,1]\ .
\end{equation}
We drop the dependence on $T$ in the notation for simplicity. 
The process $(X_h, h\in [0,1])$ is a good model for the large values of $\log |\zeta(1/2+\ii(\tau+h))|$, $h\in [0,1]$, see \cite{arguin-belius-harper,harper, sound} for more details.
For example, it is known that the deterministic order of $\max_{h\in[0,1]}X_h$ corresponds to the one in \eqref{eqn: FHK}, as proved in \cite{arguin-belius-harper}.
Roughly speaking, the process $X_h$ corresponds to the leading order of the logarithm of the Euler product \eqref{eqn: zeta} with the identification 
$$
\left(p^{-\ii \tau}, p \text{ primes}\right) \longleftrightarrow \left(U_p, p \text{ primes}\right)\ .
$$
It is easily checked by computing the joint moments that the above identification is exact as $T\to\infty$ in the sense of finite-dimensional distribution.

The covariance can be calculated using the explicit distribution of the $U_p$'s: 
\begin{equation}
\label{eqn: covariance1}
\begin{aligned}
\E[X_hX_{h'}]
&=\sum_{p\leq T}\int_0^{2\pi} \frac{1}{2}\left(e^{\ii (\theta-h\log p)}+e^{-\ii (\theta-h\log p)}\right)\cdot  \frac{1}{2}\left(e^{\ii (\theta-h'\log p)}+e^{-\ii (\theta-h'\log p)}\right)\  \frac{\rd\theta}{2\pi}\\
&= \frac{1}{2}\sum_{p\leq T} \frac{\cos(|h-h'|\log p)}{p}\ .
\end{aligned}
\end{equation}
We are interested in the correlation coefficient or {\it overlap} (in the spin glass terminology):
\begin{equation}
\label{eqn: rho}
\rho(h,h')
=\frac{\E[X_hX_{h'}]}{\sqrt{\E[X_h^2]\ \E[X_{h'}^2]}}\ , \qquad \text{ for a given pair $(h,h')$.}
\end{equation}
Any sum over primes can be estimated using the Prime Number Theorem \cite{montgomery-vaughan}, which gives the density of the primes up to very good errors,
\begin{equation}
\label{eqn: PNT}
\#\{p\leq x: p \text{ prime}\}=\int_2^x \frac{1}{\log y} \rd y+ \OO(xe^{-c\sqrt{\log x}})\ .
\end{equation}
(The error term, which is already more than sufficient for our purpose, is improved under the Riemann hypothesis.)
In particular, this can be used to rewrite the covariances as (see Lemma \ref{lem: sum primes} below for details), 
\begin{equation}
\label{eqn: covariance2}
\E[X_h^2]=\frac{1}{2}\sum_{p\leq T}p^{-1}=\frac{1}{2}\log\log T +\OO(1) \qquad \E[X_hX_{h'}]= \frac{1}{2}\log |h-h'|^{-1}+\OO(1)\ .
\end{equation}
The process $(X_h)$ is said to be {\it log-correlated}, since the covariance decays approximately like the logarithm of the distance. 
The correlation coefficients as a function of the distance become
\begin{equation}
\label{eqn: log cor}
\rho(h,h')=\frac{\log |h-h'|^{-1}}{\log\log T} +\oo(1)\ , \text{ for $|h-h'|\geq (\log T)^{-1}$.}
\end{equation}
Throughout the paper, we will use the notation $f(T)=\oo(g(T))$ if $f(T)/g(T)\to 0$ and $f(T)=\OO(g(T))$ if $f(T)/g(T)$ is bounded. We will sometimes use $f(T)\ll g(T)$ for short  if $f(T)=\OO(g(T))$ (the Vinogradov notation).

The main result of this paper is the limiting distribution of the correlation coefficient when $h$ and $h'$ are sampled from the Gibbs measure. This is referred to as the {\it two-overlap distribution} in the spin-glass terminology. We denote the the Gibbs measure by
\begin{equation}
G_{\beta,T}(A)=\int_A \frac{e^{\beta X_h}}{Z_{\beta, T}} \ \rd h \qquad Z_{\beta,T}=\int_0^1 e^{\beta X_h}\ \rd h\ .
\end{equation}

\begin{thm}
\label{thm: main}
For every $\beta>2$ and for any interval $I\subseteq [0,1]$, 
$$
\lim_{T\to \infty} \E\left[G_{\beta, T}^{\times 2}\{(h,h'): \rho(h,h')\in I\}\right]
= \frac{2}{\beta}\1_I(0) + (1-\frac{2}{\beta})\1_I(1)\ .
$$
where $\1_I$ is the indicator function of the set $I$.
In other words, when $h,h'$ are sampled independently from the Gibbs measure $G_{\beta,T}$, the random variable $\rho(h,h')$  is Bernoulli-distributed with parameter $2/\beta$ in the limit $T\to\infty$.
\end{thm}
The limit is exactly the two-overlap distribution of a 1-RSB spin glass. In view of the relation  \eqref{eqn: log cor} between the correlation coefficient and the distance $|h-h'|$, 
the result means that the large values of $X_h$ must lie at a distance $\OO(1)$ or $\OO((\log T)^{-1})$. The mesoscopic distances $(\log T)^{-\alpha}$, $0<\alpha<1$ are effectively ruled out. Similar results were obtained for the REM model \cite{derrida}, and log-correlated processes \cite{derrida-spohn,bovier-kurkova1, bovier-kurkova2,arguin-zindy1, ABK,arguin-zindy2,jagannath, ouimet2}. 

In the spirit of the Fyodorov-Hiary-Keating conjecture, Theorem \ref{thm: main} suggests that $\log |\zeta|$ exhibits 1-RSB for $\beta$ large enough.
\begin{conj}
\label{conj: zeta}
Consider 
$$
\mathcal G_{\beta}(t)=|\zeta(1/2+\ii t)|^\beta \qquad \mathcal Z_{\beta}(t)=\int_0^1 \mathcal G_{\beta}(t+h)\rd h\ .
$$
For $\beta>2$, and any interval $I\subseteq [0,1]$, if $\tau$ is sampled uniformly on $[T,2T]$:
$$
\lim_{T\to \infty} \E
\left[\int_{\{(h,h'):\rho(h,h')\in I\}} \frac{\mathcal G_{\beta}(\tau+h)\cdot \mathcal G_{\beta}(\tau+h')}{ Z_{\beta}(\tau)^2} \ \rd h\rd h'\right]
=  \frac{2}{\beta}\1_I(0) + (1-\frac{2}{\beta})\1_I(1)\ \ .
$$
In other words, points $h,h'$ whose $\zeta$-value is of the order of $\log\log T$ are at a distance of $\OO(1)$ or $\OO((\log T)^{-1})$. 
\end{conj}
The above conjecture implies a strong clustering of the high values of $\zeta$ at a scale $(\log T)^{-1}$ akin to the one observed in log-correlated process \cite{abk_genealogy}.
In turns, this phenomenon has important consequences for the joint statistics of high values which should be Poissonian at a suitable scale as for log-correlated processes \cite{abk_poisson, biskup-louidor}.
In particular, it is expected that the statistics of the Gibbs weights is Poisson-Dirichlet \cite{arguin-zindy1,arguin-zindy2}, and that the Gibbs measure converges to an atomic measure on $[0,1]$, see \cite{vargas}.
This perspective is studied in \cite{ouimet}, and will be discussed further in a forthcoming paper.  

\noindent{\bf Acknowledgements}. L.-P. A. is supported by NSF CAREER 1653602, NSF grant DMS-1513441, and a Eugene M. Lang Junior Faculty Research Fellowship.
W. T. is partially supported by  NSF grant DMS-1513441. Both authors would like to thank Fr\'ed\'eric Ouimet for useful comments on a first version of the paper. L.-P. A. is indebted to Chuck Newman for his constant support and his scientific insights throughout the years. 

\subsection{Main Propositions and Proof of the Theorem \ref{thm: main}}

The proof of Theorem \ref{thm: main} is based on a method developed for log-correlated Gaussian processes by Arguin \& Zindy \cite{arguin-zindy1, arguin-zindy2}.
It was adapted from a method of Bovier \& Kurkova  \cite{bovier-kurkova1, bovier-kurkova2} for Generalized Random Energy Models (GREM's). 
The main idea is to relate the distribution of the overlaps with the free energy of a perturbed process. 
In the present case, the process is not Gaussian and the method has to be modified.
To this aim, consider the process at {\it scale} $\alpha$, for $0<\alpha<1$, where the sum over primes is truncated at $\exp((\log T)^\alpha)$,
\begin{equation}
X_h(\alpha)=\sum_{\log p\leq (\log T)^\alpha} \frac{\Re(U_p \ p^{-ih})}{p^{1/2}}, \qquad h\in [0,1]\ .
\end{equation}
Note that $X_h(1)=X_h$.
For a small parameter $|u|<1$, we consider the {\it free energy} of the perturbed process $X_h+uX_h(\alpha)$ at scale $\alpha$:
\begin{equation}
\label{eqn: f_T}
\begin{aligned}
F_T(\beta;\alpha, u)
&=
\E\left[\log \int_{0}^1 \exp\big(\beta (X_h+uX_{h}(\alpha)\big) \rd h\right]\ .
\end{aligned}
\end{equation}
The connection between the free energy \eqref{eqn: f_T} and the distribution of the correlation coefficients is through Gaussian integration by parts.
Of course, for the process $X_h$, this step is only approximate. It follows closely the work of Carmona \& Hu \cite{carmona-hu} and Auffinger \& Chen \cite{auffinger-chen} on the universality of the free energy and overlap distributions in the Sherrington-Kirkpatrick model. 
\begin{prop}
\label{prop: overlap free}
For any $0< \alpha< 1$,
$$
\left|\int_0^\alpha \ \E \left[G_{\beta, T}^{\times 2}\big\{(h,h'): \rho(h,h')\leq y \big\}\right]\  \rd y
-\frac{2}{\beta^2\log\log T} \frac{\partial  F_T}{\partial u}(\beta;\alpha, 0) \right|=\oo(1)\ .
$$
\end{prop}
The free energy of the perturbed process is calculated using Kistler's multiscale second moment method \cite{kistler}.
The treatment is similar to the one of Arguin \& Ouimet \cite{arguin-ouimet} for the perturbed Gaussian free field.
The same result can be obtained by adapting the method of Bolthausen, Deuschel \& Giacomin \cite{bolthausen-deuschel-giacomin} and Daviaud \cite{daviaud} to the model as was done in \cite{arguin-zindy1,arguin-zindy2}.
Kistler's method is simpler and more flexible.
The result is better stated by first defining
\begin{equation}
\label{eqn: f}
f(\beta, \sigma^2)=
\begin{cases}
\beta^2\sigma^2/4 \ \ &\text{if $\beta\leq 2/\sigma$,}\\
\beta\sigma-1\ \ &\text{if $\beta\geq 2/\sigma$.}
\end{cases}
\end{equation}
\begin{prop}
\label{prop: free energy}
For every $\beta>0$ and $|u|<1$, the following limit holds
$$
\lim_{T\to\infty}\frac{1}{\log\log T} F_T(\beta;\alpha, u)
=
\begin{cases}
f\big(\beta, (1+u)^2\alpha+ (1-\alpha)\big) \ \ &\text{ if $u<0$,}\\
\alpha f(\beta, (1+u)^2) + (1-\alpha) f(\beta,1) \ \ &\text{ if $u\geq 0$.}
\end{cases}
$$
\end{prop}

The theorem follows from the above two propositions. They are proved in Sections 3 and 4 respectively.
Estimates on the model needed for the proofs are given in Section 2. 

\begin{proof}[Proof of Theorem \ref{thm: main}]
We need to show that the distribution of $\rho(h,h')$ converges weakly to $\frac{2}{\beta}\delta_0+(1-\frac{2}{\beta})\delta_1$ where $\delta_a$ stands for the Dirac measure at $a$. 
Write $x_{\beta, T}(s)$ for $\E[G_{\beta, T}^{\times 2}\big\{(h,h'): \rho(h,h')\leq s \big\}]$. By compactness of the space of probability measures on $[0,1]$, we can find a subsequence of $(x_{\beta, T})$ that converges weakly to $x_\beta$ as $T\to\infty$. We show that the limit $x_\beta$ is unique and equals $x_\beta(s)=2/\beta$ for $0\leq s <1$, thereby proving the claimed convergence. 

By definition of weak convergence, $x_{\beta,T}(s)$ converges to $x_\beta(s)$ at all points of continuity of $s$. Since $x_\beta$ is non-decreasing, this implies convergence almost everywhere. Thus, the dominated convergence theorem implies
\begin{equation}
\label{eqn: subsequence}
\lim_{T\to\infty}\int_0^\alpha x_{\beta, T}(s) \ \rd s=\int_0^\alpha x_{\beta}(s) \ \rd s\ \ \text{ , for $0<\alpha<1$.}
\end{equation}
The left-hand side can be rewritten using Proposition \ref{prop: overlap free} as
\begin{equation}
\label{eqn: derivative}
\lim_{T\to\infty}\int_0^\alpha x_{\beta, T}(s) \ \rd s
= \lim_{T\to\infty}\frac{2}{\beta^2\log\log T} \frac{\partial  F_T}{\partial u}(\beta;\alpha, 0)\ .
\end{equation}
Since $\big(((\log\log T)^{-1} F_T(\beta;\alpha, u)\big)_T$ is a sequence of convex functions of $u$, the limit of the derivatives is the derivative of the limit at any point of differentiability. 
Here the limit of the expectation of the free energy is given by Proposition \ref{prop: free energy}, for $u$ small enough so that $\beta>2/\sigma$ whenever $\beta >2$,
\begin{equation}
\label{eqn: derivative free}
\begin{aligned}
\lim_{T\to\infty}\frac{1}{\log\log T} F_T(\beta;\alpha, u)
&=
\begin{cases}
\beta\Big((1+u)^2\alpha+ (1-\alpha)\Big)^{1/2}-1 \ \ &\text{ if $u<0$,}\\
\alpha \beta (1+u) +(1- \alpha) \beta -1 \ \ &\text{ if $u\geq 0$.}
\end{cases}
\end{aligned}
\end{equation}
In particular, the expected free energy is differentiable at $u=0$. Therefore, equations \eqref{eqn: subsequence}, \eqref{eqn: derivative} and  \eqref{eqn: derivative free} altogether imply
$$
\int_0^\alpha x_{\beta}(s) \ \rd s= \alpha \frac{2}\beta\  \ \text{ , for $0<\alpha<1$.}
$$
This means that for any $0<\alpha<\alpha'<1$ we have
$$
\frac{1}{\alpha'-\alpha}\int_\alpha^{\alpha'} x_{\beta}(s) \ \rd s= \frac{2}\beta\ .
$$
By taking $\alpha'-\alpha\to 0$, we conclude from the Lebesgue differentiation theorem that $x_{\beta}(s)=2/\beta$ almost everywhere. Since $x_\beta$ is non-decreasing and right-continuous, this implies that $x_\beta(s)=2/\beta$ for every $0\leq s<1$ as claimed. 
\end{proof}

\section{Estimates on the model of zeta}

In this section, we gather the estimates on the model of zeta needed for the proof of Propositions \ref{prop: overlap free} and \ref{prop: free energy}.
Most of these results are contained in \cite{arguin-belius-harper}. 
We include them for completeness since we will need to deal with a perturbed version of the process $(X_h)$.
It is is important to point out that most (but not all!) of these estimates can be obtained for zeta itself with some more work, see \cite{ABBRS}. 

The essential input from number theory for the model is the Prime Number Theorem \eqref{eqn: PNT}. 
It shows that the density of the primes is approximately $1/\log p$. 
This implies, for example, that $\sum_{p} p^{-a}<\infty$ for $a>1$. 
The equation \eqref{eqn: covariance2} expressing the log-correlations for $h\neq h'$ is straightforward from the following lemma by taking $\Delta=|h-h'|$ and by splitting the sum \eqref{eqn: covariance1} into the ranges $\log p\leq |h-h'|^{-1}$ and $|h-h'|^{-1}<\log p\leq \log T$.
\begin{lem}
\label{lem: sum primes}
Let $2\leq P<Q<\infty$. Then for $\Delta>0$, we have
\begin{equation}
\label{eqn: cos p}
\begin{aligned}
\sum_{P\leq p\text{ primes}\leq Q} \frac{\cos (\Delta\cdot \log p)}{p}&= \int_P^Q \frac{\cos (\Delta\cdot \log v)}{v\log v} \rd v+ \OO(e^{-c\sqrt{\log P}})\\
&=
\begin{cases}
 \log\log Q -\log\log P +\OO(1) \ & \text{ for $\Delta\cdot \log Q\leq 1$,}\\
 \OO(\frac{1}{\Delta\cdot \log P})+ \OO(e^{-c\sqrt{\log P}}) & \text{ for $\Delta\cdot\log P\geq  1$.}
\end{cases}
\end{aligned}
\end{equation}
\end{lem}

\begin{proof}
Denote by ${\rm Li}(x)=\int_2^x \frac{1}{\log y} \rd y$ the logarithmic integal.
Write $\mathcal E (x)$ for the function of bounded variation $\pi(x)-{\rm Li}(x)$ giving the error, and $f(x)$ for $\frac{\cos(\Delta\cdot \log x)}{x}$. Clearly, we have
$$
\sum_{P\leq p\leq Q} f(p )=\int_P^Q f(x) \pi(dx)= \int_P^Q \frac{f(x)}{\log x} dx + \int_P^Q f(x) \mathcal E(dx)\ .
$$
It remains to estimate the error term. By integration by parts,
$$
\int_P^Q f(x) \mathcal E(dx)=
 \mathcal E(Q ) f(Q )-  \mathcal E(P ) f(P ) - \int_P^Q \mathcal E(x) f'(x) dx\ .
$$
Note that $f(x)$ is of the order of $1/x$ and $f'(x)$ is of the order of $1/x^2$. Since $\mathcal E(x)=O(x e^{-c\sqrt{\log x}})$, the first claimed equality follows.
For the dichotomy in the second equality,  in the case $\Delta\cdot\log Q\leq 1$, we expand the cosine to get after the change of variable $y=\log x$
$$
 \int_{ P}^{Q} \frac{f(x)}{\log x} \rd x=\int_{\log P}^{\log Q} \frac{\cos(\Delta\cdot y)}{y} \rd y=\int_{\log P}^{\log Q} \left(\frac{1}{y}+\OO(\Delta^2\cdot y)\right)\  \rd y\ .
$$
The result follows by integration. In the case $\Delta\cdot \log P \geq 1$, we integrate by parts to get
$$
\int_P^Q \frac{f(x)}{\log x} dx=\frac{\sin(\Delta\cdot y)}{\Delta\cdot y}\Big|_{\log P}^{\log Q}+\int_{\log P}^{\log Q} \frac{\sin(\Delta\cdot y)}{\Delta\cdot y^2} \rd y\ .
$$
Both terms are  $\OO(\frac{1}{\Delta\cdot \log P})$ as claimed. 
\end{proof}

Proposition \ref{prop: free energy} gives an expression for the free energy \eqref{eqn: f_T} of the perturbed process at scale $\alpha$.
For simplicity, we denote this process by
\begin{equation}
\label{eqn: X tilde}
\widetilde X_h= (1+u)X_h(\alpha) + X_h(\alpha,1) \  \ \text{for $X_h(\alpha,1)=X_h-X_h(\alpha)$, $h\in [0,1]$.}
\end{equation}
 Note that we recover $X_h$ at $u=0$.
The finite-dimensional distributions of $(\widetilde X_h)$ can be explicitly computed.
In fact, it is not hard to compute explicitly the moment generating function for any increment of $(X_h)$.
We will only need the two-dimensional case.
\begin{prop}
\label{prop: MGF}
Let $0\leq \alpha_1<\alpha_2\leq 1$. Consider $X_h(\alpha_1,\alpha_2)=X_h(\alpha_2)-X_h(\alpha_1)$. We have for $\lambda, \lambda'\in \R$ and $h,h'\in [0,1]$,
$$
\begin{aligned}
&\E\left[\exp\left(\lambda X_h(\alpha_1,\alpha_2) + \lambda'X_{h'}(\alpha_1,\alpha_2\right)\right]\\
&\hspace{2cm}=C(\lambda,\lambda')\cdot \exp\left(\frac{1}{2}\sum_{\substack{\log p\leq (\log T)^{\alpha_2}\\ \log p>(\log T)^{\alpha_1} }}
\frac{1}{2p}\Big(\lambda^2+{\lambda'}^2+2\lambda\lambda'\cos(|h-h'|\log p)\Big)\right)\ ,
\end{aligned}
$$ 
where $C=C(\lambda, \lambda')$ is bounded if $\lambda$ and $\lambda'$ are bounded uniformly in $T$. 
\end{prop}
\begin{proof}
The expression can be evaluated explicitly as follows. Since the $U_p$'s are independent, we can first restrict the computation to a single $p$.
Straightforward manipulations yield
$$
\E\left[\exp\left(p^{-1/2}\lambda \cdot \Re(U_pp^{-\ii h})+ p^{-1/2}\lambda'\cdot\Re(U_pp^{-\ii h'})\right)\right]
=\E\left[\exp(aU_p+\bar{a}\overline U_p)\right]
$$
for $a=(2p^{1/2})^{-1}(\lambda p^{-\ii h}+\lambda' p^{-\ii h'})$. By expanding the exponentials and using the fact that $U_p$ is uniform on the unit circle, we get
\begin{equation}
\label{eqn: mgf p}
\begin{aligned}
\E\left[\exp(aU_p+\bar{a}\overline U_p)\right]
&= \sum_{n=0}^\infty\sum_{k=0}^n  \frac{a^k\bar{a}^{n-k}}{n!}{n\choose k}\E[U_p^k\overline U_p^{n-k}]\\
&=\sum_{m=0}^\infty  \frac{1}{(m!)^2}\left(\frac{\lambda^2+{\lambda'}^2+2\lambda\lambda'\cos(|h-h'|\log p)}{4p}\right)^m\\
&= 1+ \left(\frac{\lambda^2+{\lambda'}^2+2\lambda\lambda'\cos(|h-h'|\log p)}{4p}\right)+\OO(p^{-2})\ ,
\end{aligned}
\end{equation}
where the $\OO$-term depends on $\lambda,\lambda'$. The second equality follows from the fact that the expectation is non-zero only if $k=n/2$. 
It remains to take the product over the range of $p$.
The claim then follows from the fact that the sum of $p^{-2}$ is finite by \eqref{eqn: PNT}.
\end{proof}
Proposition \ref{prop: MGF} yields Gaussian bounds in the large deviation regime we are interested in. 
Indeed, by Chernoff's bound (optimizing over $\lambda$), it implies that, for $\gamma>0$, 
\begin{equation}
\label{eqn: increment one point}
\P\left(X_h(\alpha_1,\alpha_2)>\gamma\log\log T\right)
\ll \exp\left(-\frac{\gamma^2 \log\log T}{(\alpha_2-\alpha_1)}\right)=(\log T)^{\frac{-\gamma^2}{\alpha_2-\alpha_1}}\ ,
\end{equation}
where we used Lemma \ref{lem: sum primes} to estimate the sum over primes.
This supports the heuristic that $X_h(\alpha_1,\alpha_2)$ is approximately Gaussian of variance $\frac{\alpha_2-\alpha_1}{2}\log\log T$. 
This implies for $\widetilde X_h$ in \eqref{eqn: X tilde}
\begin{equation}
\label{eqn: X tilde one-point}
\P\left(\widetilde X_h>\gamma\log\log T\right)
\ll \exp\left(-\frac{\gamma^2 \log\log T}{(1+u)^2\alpha+(1-\alpha)}\right)=(\log T)^{\frac{-\gamma^2}{(1+u)^2\alpha+(1-\alpha)}}\ .
\end{equation}
The same can be done for two points $h,h'$. Using Lemma \ref{lem: sum primes} again, we get
\begin{equation}
\label{eqn: increment two points}
\begin{aligned}
&\P\left(X_h(\alpha_1,\alpha_2)>\gamma\log\log T,X_{h'}(\alpha_1,\alpha_2)>\gamma \log\log T \right)\\
&\hspace{3cm}\ll 
\begin{cases}
\exp\left(-\frac{\gamma^2 \log\log T}{(\alpha_2-\alpha_1)}\right) \ &\text{ if $|h-h'|\leq (\log T)^{-\alpha_2}$,}\\
\exp\left(-2\frac{\gamma^2 \log\log T}{(\alpha_2-\alpha_1)}\right)\ &\text{ if $|h-h'|\geq (\log T)^{-\alpha_1}$.}\\
\end{cases}
\end{aligned}
\end{equation}
This can be interpreted as follows. The increments are (almost) independent if the distance between the points is larger than the relevant scales of the increments, and are (almost) perfectly correlated if the distance is smaller than the scales. 

It is important to note that if $\alpha_1>0$, then a stronger estimate than the one of Proposition \ref{prop: MGF} holds. 
This is because the sum over primes in \eqref{eqn: mgf p} is then negligible since it is the tail of a summable series. 
This means that the constant $C(\lambda,\lambda')$ is then $1+\OO(1)$. 
This gives a precise Gaussian estimate by inverting the moment generating function (or the Fourier transform if we pick $\lambda,\lambda'\in \mathbb C$).
We omit the proof for conciseness and we refer to \cite{arguin-belius-harper} where this is done using a general version of the Berry-Esseen theorem.
\begin{prop}[see Propositions 2.9, 2.10, 2.11 in \cite{arguin-belius-harper}]
\label{prop: berry-esseen}
For $0<\alpha_1<\alpha_2\leq 1$ and $0<\gamma <1$, we have for $h\in [0,1]$,
$$
\P(X_{h}(\alpha_1,\alpha_2)>\gamma \log\log T)
\gg\frac{1}{\sqrt{\log\log T}}\exp\left(-\frac{\gamma^2 \log\log T}{(\alpha_2-\alpha_1)}\right)=(\log T)^{\frac{-\gamma^2}{\alpha_2-\alpha_1}+\oo(1)}\ .
$$
Moreover, if $|h-h'|>(\log T)^{-\alpha_1}$, then
$$
\P(X_{h}(\alpha_1,\alpha_2)>\gamma \log\log T, X_{h'}(\alpha_1,\alpha_2)>\gamma \log\log T)
=(1+\oo(1)) {\P(X_{h}(\alpha_1,\alpha_2)>\gamma \log\log T)}^2\ .
$$
\end{prop}

Since the process $(X_h,h\in[0,1])$ is continuous and not discrete, we need a last estimate to control all values in an interval of length corresponding to the relevant scale.
This is needed when proving rough bound on the maximum in Lemma \ref{lem: max}.
Heuristically, it says that the maximum of $X_h(\alpha_1,\alpha_2)$ over an interval of width smaller than $(\log T)^{-\alpha_2}$ behaves like a single value $X_h(\alpha_1,\alpha_2)$. This is done in \cite{arguin-belius-harper} by a chaining argument and we omit the proof for conciseness.
\begin{lem}[Corollary 2.6 in \cite{arguin-belius-harper}]
\label{lem: smoothing}
Let $0\leq\alpha_1<\alpha_2\leq 1$. For every $h\in[0,1]$ and $\gamma>0$, we have
$$
\P\left(\max_{|h-h'|\leq (\log T)^{-\alpha_2}}X_{h'}(\alpha_1,\alpha_2)> \gamma\log\log T\right) \ll (\log T)^{-\frac{\gamma^2}{\alpha_2-\alpha_1}}\ .
$$
In particular, we have
$$
\P\left(\max_{|h-h'|\leq (\log T)^{-1}}\widetilde X_{h'}> \gamma\log\log T\right) \ll (\log T)^{-\frac{\gamma^2}{\alpha(1+u)+(1-\alpha)}}\ .
$$
\end{lem}

\section{Proof of Proposition \ref{prop: overlap free}}
As mentioned in Section 1.3, the proof of Proposition \ref{prop: overlap free} is based on an approximate Gaussian integration by parts as in \cite{carmona-hu} and \cite{auffinger-chen}.
The following lemma is an adaptation for complex random variables of  Lemma 4 in \cite{carmona-hu} .
\begin{lem}
\label{lem: by parts}
 Let $\xi$ be a complex random variable such that $\E[|\xi|^3]<\infty$, and $\E[\xi^2]=\E[\xi]=0$. Let $F:\mathbb{C}\to\mathbb C$
be a twice continuously differentiable function such that for some $M>0$, 
$$
\left\Vert\partial_z^2F\right\Vert_{\infty},
\left\Vert \partial^2_{\overline z}F\right\Vert_\infty <M\ ,
$$ 
where  $\left\Vert f\right\Vert_\infty=\sup _{z\in \mathbb{C}}\left\vert
f(z,\overline{z})\right\vert$.
Then
$$
\left|\E\left[\xi F(\xi,\overline{\xi})\right] - \E[\left\vert\xi\right\vert^2]\ 
\E\left[ \partial _{\overline{z}}F(\xi,\overline{\xi})\right] \right|
\ll M\ \E[|\xi|^3].
$$
\end{lem}
\begin{proof}
Since $\E[\xi^2]=\E[\xi]=0$, the left-hand side can be written as
\begin{equation}
\label{eqn: to prove by parts}
\E\left[\xi \big(F(\xi,\overline{\xi}) -F(0,0)- \xi \partial_z F(0,0) - \overline \xi \partial_{\overline z} F(0,0)\big)\right] 
 - \E[\left\vert\xi\right\vert^2]\
\E\left[\Big(\partial _{\overline{z}}F(\xi,\overline{\xi})- \partial _{\overline{z}}F(0,0)\Big)\right]\ .
\end{equation}
By Taylor's theorem in several variables and the assumptions, the following estimates hold
$$
\begin{aligned}
&\left| F(\xi,\overline{\xi}) -F(0,0)- \xi \partial_z F(0,0) - \overline \xi \partial_{\overline z} F(0,0) \right|\ll M|\xi|^2\\
&\left| \partial _{\overline{z}}F(\xi,\overline{\xi})- \partial _{\overline{z}}F(0,0) \right|\ll M|\xi|\ .
\end{aligned}
$$
Therefore the norm of \eqref{eqn: to prove by parts} gives
$$
\left|\E\left[\xi F(\xi,\overline{\xi})\right] - \E[\left\vert\xi\right\vert^2]\ \E\left[ \partial _{\overline{z}}F(\xi,\overline{\xi})\right] \right|
\ll
M (\E[|\xi|^3]+ \E[|\xi|^2]\E[|\xi|])\ .
$$
The claim then follows by H\"older's inequality. 
\end{proof}

As in \cite{carmona-hu}, the lemma can be applied to relate the derivative of the free energy to the two-point correlations of the process.
\begin{prop}
\label{prop: carmona-hu}
For any $p\leq T$, we have
$$
\begin{aligned}
&\frac{\partial}{\partial u} \E\left[ \log \int_0^1 \exp\big(\beta (X_h(T)+u \ \Re(U_pp^{-ih-1/2})\big) \rd h\right]\Big|_{u=0}\\
&\hspace{2cm}= \frac{\beta}{2}\E\left[ \int_{[0,1]^2} \frac{1- \cos(|h-h'|\log p)}{p}\ \rd G^{\times 2}_{\beta, T}(h,h')\right]
+\OO(p^{-3/2})\ .
\end{aligned}
$$
\end{prop}

\begin{proof}
Write for short $\omega_p(h)=(2p^{1/2})^{-1}p^{-\ii h}$. Direct differentiation yields at $u=0$
\begin{equation}
\label{eqn: U_p integral}
\int_0^1 U_p \omega_p(h)\ \rd G_{\beta, T}(h) + \int_0^1 \overline{U_p} \overline{\omega_p}(h)\ \rd G_{\beta, T}(h)\ .
\end{equation}
We make the dependence on $U_p$ in the measure $G_{\beta, T}$ explicit. For this, define 
$$
Y_p(h)=\beta \sum _{\substack{q \le T\\p\neq q}}\Re \left(\frac{U_q q^{-ih}}{q^{1/2}} \right)\ .
$$
Clearly, $Y_p(h)$ is independent of $U_p$ by definition. Consider
$$
F_p(z,\overline{z})=\frac{\int_0^1 \omega_p(h)\ \exp(\beta \omega_p(h) z  +\beta \overline{\omega_p}(h)
\overline{z}+ Y_p(h))\ \rd h}
{\int_0^1\exp(\beta \omega_p(h') z  +\beta \overline{\omega_p}(h')\overline z+ Y_p(h'))\ \rd h'}\ .
$$
Note that with this definition, the first integral in \eqref{eqn: U_p integral} is $U_pF_p(U_p, \overline U_p)$ and the second is its complex conjugate. 
This shows that the derivative of the expectation at $u=0$ is 
$$
\E\left[U_p\cdot F_p(U_p, \overline U_p)+\overline{U_p}\cdot \overline{F_p}(U_p, \overline U_p)\right]\ .
$$
It remains to apply Lemma \ref{lem: by parts} with the function $F_p(z,\overline z)$ and $\xi=U_p$. Write for short for a function $H$ on $[0,1]$
$$
\left\langle  H\right\rangle^{(z,\overline{z})}=
 \frac{\int_0^1 H(h)\ e^{\beta\left( z\ \omega_p(h) +\overline{z}\ \overline{\omega_p}(h)\right)+ Y_p(h)} \rd h}
 {\int_0^1 e^{\beta\left( z\ \omega_p(h) +\overline{z}\ \overline{\omega_p}(h)\right)+ Y_p(h)} \rd h}\ .
$$
Direct differentiation of the above yields
\begin{equation}
\label{eqn: H deriv}
\begin{aligned}
 \partial_{\overline z} \left\langle  H\right\rangle^{(z,\overline{z})}=
\beta\left(\left\langle  H \overline{\omega_p}\right\rangle^{(z,\overline{z})}-\left\langle
 H\right\rangle^{(z,\overline{z})}\left\langle  \overline{\omega_p}\right\rangle^{(z,\overline{z})}\right)\ .
\end{aligned}
\end{equation}
In particular, for $H=\omega_p$, we get
\begin{equation}
\label{eqn: F deriv}
\partial_{\overline z}F_p(z,\overline z)=\beta\left(\left\langle   |\omega_p|^2\right\rangle^{(z,\overline{z})}-
|\left\langle \omega_p\right\rangle^{(z,\overline{z})}|^2\right)\ .
\end{equation}
When evaluated at $z=U_p$, this is by definition of $\omega_p$
\begin{equation}
\label{eqn: F deriv 2}
\partial_{\overline z}F_p(U_p,\overline U_p)=\frac{\beta}{4} \int_{[0,1]^2}(p^{-1}-p^{-1}\cos(|h-h'|\log p)) \ \rd G^{\times 2}_{\beta, T}(h,h')\ .
\end{equation}
Clearly, $|\omega_p|\leq p^{-1/2}$. Therefore the second derivatives are easily checked to be bounded by $\OO(p^{-3/2})$ by applying the formula \eqref{eqn: H deriv} to each term of \eqref{eqn: F deriv}. 
The statement of the lemma then follows from Lemma \ref{lem: by parts} and \eqref{eqn: F deriv 2}, after noticing that the second term of \eqref{eqn: U_p integral} is the conjugate of the first. 
\end{proof}

The proof of Proposition \ref{prop: overlap free} is an application of Proposition \ref{prop: carmona-hu} to a range of primes.
\begin{proof}[Proof of Proposition \ref{prop: overlap free}]
Recall the definition of $\rho(h,h')$ in equations \eqref{eqn: rho} and \eqref{eqn: log cor}.
On one hand, Fubini's theorem directly  implies that
\begin{equation}
\label{eqn: fubini}
\begin{aligned}
\int_0^\alpha \ G_{\beta, T}^{\times 2}\big\{(h,h'): \rho(h,h')\leq y \big\}\  \rd y
&=\int_{[0,1]^2} \left(\int_0^\alpha {\1}_{\{\rho(h,h')\leq r\}} \ \rd r\right) \rd G_{\beta, T}(h,h')\\
&=\int_{[0,1]^2} \left(\alpha-\rho(h,h')\right)\1_{\{\rho(h,h')\leq \alpha\}} \rd G_{\beta, T}(h,h')\ .
\end{aligned}
\end{equation}
It remains to check on the other hand that the derivative in the proposition is close to the expectation of the above. Direct differentiation of \eqref{eqn: f_T} at $u=0$ yields by Proposition \ref{prop: carmona-hu}
\begin{equation}
\label{eqn: deriv F}
 \frac{\partial  F_T}{\partial u}(\beta;\alpha, 0)
 = \frac{\beta^2}{2} \int_{[0,1]^2} \sum_{\log p\leq (\log T)^{\alpha}} \E\big[p^{-1}(1- \cos(|h-h'|\log p))\ \rd G^{\times 2}_{\beta, T}(h,h')\big]
+\OO(\sum_{ p\leq e^{(\log T)^{\alpha}}}p^{-3/2})\ .
 \end{equation}
 The error term is of order one by \eqref{eqn: PNT}.
 Similarly, if $|h-h'|\leq (\log T)^{-\alpha}$, the sum in the integral is by \eqref{eqn: cos p}
 $$
 \sum_{\log p\leq (\log T)^{\alpha}}p^{-1}(1- \cos(|h-h'|\log p))=
 \alpha \log \log T - \alpha\log\log T +\OO(1)=\OO(1)\ .
 $$
On the other hand, if $|h-h'|>(\log T)^{-\alpha}$, the sum can be divided into three parts
$$
\sum_{\log p\leq (\log T)^{\alpha}}p^{-1}
-\sum_{\log p\leq |h-h'|^{-1}}p^{-1} \cos(|h-h'|\log p)
- \sum_{|h-h'|^{-1}<\log p\leq (\log T)^{\alpha}}p^{-1}\cos(|h-h'|\log p)\ .
$$
When equation \eqref{eqn: cos p} is applied to each of the parts, this equals
$$
\alpha\log \log T-\log |h-h'|^{-1} +\OO(1)\ .
$$
Furthermore, recall from \eqref{eqn: log cor} that $\rho(h,h')\log\log T$ differs from $\log |h-h'|^{-1}$ by $\oo(\log\log T)$. 
This implies that the conditions on $\log |h-h'|^{-1}$ can be replaced by $\rho(h,h')\log\log T$ at a cost of a term $\oo(\log\log T)$ (since the sum would differ by a range of $\log p$ of at most $\oo(\log T)$ primes). 
All these observations together imply
$$
 \frac{1}{\log\log T}\sum_{\log p\leq (\log T)^{\alpha}}p^{-1}(1- \cos(|h-h'|\log p))=
\left(\alpha-\frac{\log |h-h'|^{-1}}{\log\log T}\right)\1_{\{\rho(h,h')\leq \alpha\}} +\oo(1)\ .
$$
We finally conclude by putting the right side back in the integral of \eqref{eqn: deriv F} and by using \eqref{eqn: log cor} that
$$
 \frac{2}{\beta^2\log\log T}\frac{\partial  F_T}{\partial u}(\beta;\alpha, 0)
 = \int_{[0,1]^2}\left(\alpha-\rho(h,h')\right)\E\big[\1_{\{\rho(h,h')\leq \alpha\}} \rd G^{\times 2}_{\beta, T}(h,h')\big] +\oo(1)\ .
$$
 This matches the first claim \eqref{eqn: fubini} by an error $\oo(1)$ thereby proving the proposition. 
\end{proof}

\section{Proof of Proposition \ref{prop: free energy}}
We write $\widetilde X_h= (1+u)X_h(\alpha) + X_h(\alpha,1)$ as in equation \eqref{eqn: X tilde}. 
The limit of the free energy of this process is obtained by Laplace's method once the measure of high points is known, cf.~Lemma \ref{lem: high points}.
The proof of Lemma \ref{lem: high points} is based on a similar computation of \cite{arguin-ouimet} for the two-dimensional Gaussian free field based on Kistler's multiscale  second moment method \cite{kistler}. But first, we need an {\it a priori} restriction on the maximum of the process $(\widetilde X_h)$. 
The maximum depends on the value of the parameter $u$ as expected from GREM models.
With this in mind, we define
\begin{equation}
\label{eqn: gamma star}
\gamma^\star=
\begin{cases}
\Big((1+u)^2 \alpha + (1-\alpha)\Big)^{1/2}\ \ &\text{if $u\leq0$,}\\
(1+u)\alpha + (1-\alpha)\ \ &\text{if $u>0$.}
\end{cases}
\end{equation}
Note that the two expressions are equal to $1$ at $u=0$ and that $\gamma^\star>1$ if $u> 0$, and $\gamma^\star<1$ if $u<0$. The next lemma bounds the maximum of $\widetilde X_h$.
\begin{lem}
\label{lem: max}
For any $\e>0$, 
$$
\lim_{T\to\infty}\P\left(\max_{h\in[0,1]} \widetilde X_h>(1+\e)\gamma^\star \log\log T \right)=0\ .
$$
\end{lem}
\begin{proof}
This is a consequence of Lemma \ref{lem: smoothing} which shows that the large values of $X_h(\alpha)$ are well approximated by points at a distance $(\log T)^{-\alpha}$.
In the case $u\leq 0$, we use the lemma with $\alpha=1$. Without loss of generality, suppose that $\log T$ is an integer and 
consider $I_k$, $k\leq \log T$, a collection of intervals of length $(\log T)^{-1}$ that partitions $[0,1]$.
Then a simple union bound yields
$$
\begin{aligned}
\P\left(\max_{h\in[0,1]} \widetilde X_h>(1+\e)\gamma^\star \log\log T \right)
&\leq \sum_{k=1}^{\log T} \P\left(\max_{h\in I_k} \widetilde X_h >(1+\e)\gamma^\star \log\log T \right)\ .
\end{aligned}
$$
Lemma \ref{lem: smoothing} applied to $\widetilde X_h$ then implies
$$
\begin{aligned}
\P\left(\max_{h\in[0,1]} \widetilde X_h>(1+\e)\gamma^\star \log\log T \right)
&\ll (\log T) \exp\left(-\frac{(1+\e)^2 \Big((1+u)^2\alpha + (1-\alpha)\Big)\log\log T)}{(1+u)^2\alpha + (1-\alpha)}\right)\\
&\leq (\log T)^{1-(1+\e)^2}\ ,
\end{aligned}
$$
which goes to $0$ as claimed.

In the case $u>0$, an extra restriction is needed since the large values of $X_h(\alpha)$ are themselves limited. 
Proceeding as above, without loss of generality, assume that $(\log T)^\alpha$, $(\log T)^{1-\alpha}$ and $\log T$ are integers.
Consider the collection of intervals $J_j$, $j\leq (\log T)^{\alpha}$, that partitions $[0,1]$ into intervals of length $(\log T)^{-\alpha}$.
Each $J_j$ is again partitioned into intervals $I_{jk}$, $k\leq (\log T)^{1-\alpha}$, of length $(\log T)^{-(1-\alpha)}$. 
Then Lemma \ref{lem: smoothing} implies
\begin{equation}
\label{eqn: max interval}
\P\left(\exists j: \max_{h\in J_{j}} X_h(\alpha)>(1+\e)\cdot\alpha \log\log T \right)\to 0\ .
\end{equation}
Therefore, the probability of the maximum of $\widetilde X_h$ can be restricted as follows:
$$
\begin{aligned}
&\P\left(\max_{h\in[0,1]} \widetilde X_h>(1+\e)\gamma^\star \log\log T \right)\\
&= \P\left(\exists h\in [0,1]: \widetilde X_h>(1+\e)\gamma^\star \log\log T,X_h(\alpha)\leq (1+\e)\cdot \alpha \log \log T\right)+\oo(1)\\
&\ll\sum_{j,k} 
\sum_{q=0}^{(1+\e)\cdot \alpha\log\log T}
\P\left(\max_{h\in J_j} X_h(\alpha)\in[q,q+1], \max_{h\in I_{jk}} X_h(\alpha,1)>(1+\e)\gamma^\star \log\log T-(1+u)(q+1)\right)\ .
\end{aligned}
$$
The last inequality is obtained by a union bound on the partition $(I_{jk})$ and by splitting the values of the maximum of $X_h(\alpha)$ on the range $[0,(1+\e)\alpha \log\log T]$. (Note that $X_h(\alpha)$ is symmetric thus the maximum is greater than $0$ with large probability.) By independence between $(X_h(\alpha),h\in[0,1])$ and $(X_h(\alpha,1),h\in[0,1])$, Lemma \ref{lem: smoothing} can be applied twice to get the following bound on the summand:
\begin{equation}
\label{eqn: prob u<=0}
\ll \exp\left(\frac{-q^2}{\alpha \log\log T}\right)\cdot \exp\left(\frac{-\Big((1+\e)\gamma^\star\log\log T-(1+u)(q+1)\Big)^2}{(1-\alpha) \log\log T}\right)\ .
\end{equation}
On the interval $[0,(1+\e)\alpha\log\log T]$, this is maximized at the endpoint $q=(1+\e) \alpha\log\log T$. (This is where the case $u<0$ differs, as the optimal $q$ there  is within the interval. See Remark \ref{rem: dicho} for more on this.) Putting this back in \eqref{eqn: prob u<=0} and summing over $j,k$, and $q$ finally give the estimate:
$$
\begin{aligned}
&\P\left(\max_{h\in[0,1]} \widetilde X_h>(1+\e)\gamma^\star \log\log T \right)\\
&\ll (\log\log T)\cdot (\log T)^{\alpha} \cdot e^{\frac{-\Big((1+\e)\alpha \log\log T\Big)^2}{\alpha \log\log T}}\cdot (\log T)^{1-\alpha} \cdot e^{\frac{-\Big((1+\e)(1-\alpha) \log\log T\Big)^2}{(1-\alpha) \log\log T}}\\
&\ll (\log\log T)\cdot \log T^{1-(1+\e)^2}=\oo(1)\ .
\end{aligned}
$$
This concludes the proof of the lemma.
\end{proof}

Consider for $0<\alpha\leq 1$ and $|u|<1$ the (normalized) log-measure of $\gamma$-high points
\begin{equation}
\label{eqn: log-measure}
\mathcal E_{\alpha,u}(\gamma;T)=
\frac{1}{\log\log T}\log {\rm Leb}\{h\in [0,1]: \widetilde X_h> \gamma \log \log T\}\ , 0< \gamma <\gamma^\star\ .
\end{equation}
The limit of these quantities in probability can be computed following \cite{arguin-ouimet}.
\begin{lem}
\label{lem: high points}
The limit $\mathcal E_{\alpha,u}(\gamma)=\lim_{T\to \infty}\mathcal E_{\alpha,u}(\gamma;T)$ exists in probability.
We have for $u<0$,
$$
\mathcal E_{\alpha,u}(\gamma)
=-\frac{\gamma^2}{(1+u)^2 \alpha+(1-\alpha)}\ ,
$$
and for $u\geq 0$, 
$$
\mathcal E_{\alpha,u}(\gamma)
= 
\begin{cases}
-\frac{\gamma^2}{(1+u)^2 \alpha+(1-\alpha)} \ \ &\text{if $\gamma<\gamma_c$}\\
-\alpha -\frac{(\gamma-(1+u)\alpha)^2}{(1-\alpha)}\ \ &\text{if $\gamma\geq \gamma_c$.}
\end{cases}
\ \text{ for $\gamma_c=\frac{(1+u)^2\alpha+(1-\alpha)}{(1+u)}$.}
$$
\end{lem}

\begin{rem}
\label{rem: dicho}
{\rm 
The dichotomy in the log-measure is due to the fact that for $h$ with values beyond $\gamma_c\log\log T$, the intermediate values $X_h(\alpha)$ is restricted by the maximal level $\alpha\log\log T$. 
More precisely, consider
\begin{equation}
\label{eqn: M}
\begin{aligned}
\mathcal{M}_T = & {\rm
Leb}\{h\in [0,1]: \widetilde X_h> \gamma \log
\log T\}\\
\mathcal{M}'_T   =&\mathrm{Leb}\{h\in[0,1]:(1+u) X_h(\alpha) \ge
\lambda \log \log T\}\\\mathcal{M} ''_T   =&\mathrm{Leb}\{h\in[0,1]:(1+u)
X_h(\alpha) \ge
\lambda \log \log T,X_h(\alpha,1)\ge
(\gamma-\lambda)\log \log T\}\ .
\end{aligned}
\end{equation}
Clearly, we must have \(\mathcal{M}_T''\leq\mathcal{M}_T\). 
It turns out that \(\mathcal{M}_T''\) and \(\mathcal{M}_T\) are comparable for an optimal choice of $\lambda$ given by, when \(u<0 \),
\begin{equation}
\label{eqn: lambda1}
\lambda^\star =\frac{\gamma (1+u)^2\alpha}{(1+u)^2\alpha+  1-\alpha},\qquad
\gamma<\gamma^\star,
\end{equation}
and when  \( u>0\),
\begin{equation}
\label{eqn: lambda2}
\lambda^\star =\begin{cases}\frac{\gamma (1+u)^2\alpha}{(1+u)^2\alpha+
 1-\alpha}
& \text{if }0<\gamma<\gamma_c\ ,  \\
(1+u)\alpha\ & \text{if }\gamma_c\le \gamma<\gamma^\star.
\end{cases}
\end{equation}
One can see this at a heuristic level by considering first moments. Since the maximum of \(X_{h}\) is well approximated by the maximum over
lattice points spaced \((\log T)^{-1} \) apart, 
there should be \(\gamma\)-high points only if 
\begin{equation}
(\log T)\cdot \mathcal{M}_T''\ge 1.\label{eq:GammaHighPtsExist}
\end{equation}
Moreover, we have that if
\(\mathcal{M}_T'=0\), then \(\mathcal{M}_T''=0\). And the
maximum of \(X_{h}(\alpha)\) is well approximated by the maximum over
lattice points spaced \((\log T)^{-\alpha} \) apart, 
so there should be \(\gamma\)-high points only if\begin{equation}
(\log T)^{\alpha}\cdot \mathcal{M}_T'\ge 1.\label{eq:GammaHighPtsExistAlpha}
\end{equation}
Since $X_h(\alpha)$ and $X_h(\alpha,1)$ are approximately Gaussian with variance $\frac{1}{2}\alpha\log\log T$ and $\frac{1}{2}(1-\alpha)\log\log T$, the following should hold approximately
 \begin{align*}
 &\frac{\log  \E[ (\log T)^\alpha\cdot \mathcal{M}_T']}{\log\log T}=\alpha-\frac{\lambda^2}{(1+u)^2\alpha}+\oo(1)
 \\
 &\frac{\log \E[(\log T)\cdot \mathcal{M}_T'']}{\log\log T}=1-\frac{\lambda^2}{(1+u)^2\alpha}-\frac{(\gamma-\lambda)^2}{1-\alpha}+\oo(1)
\end{align*}
Together with conditions (\ref{eq:GammaHighPtsExist}) and (\ref{eq:GammaHighPtsExistAlpha}), we obtain constraints on the value of \(\lambda\): 
\begin{align}
 & \alpha-\frac{\lambda^2}{(1+u)^2\alpha}\ge 0 \ ,
 \label{eq:MaximizationConstraintAlpha}
 \\
 &1-\frac{\lambda^2}{(1+u)^2\alpha}-\frac{(\gamma-\lambda)^2}{1-\alpha}\ge0 \ .
\label{eq:MaximizationConstraint}
\end{align}
By maximizing \(\mathcal{M}_T''\), under the constraints (\ref{eq:MaximizationConstraintAlpha})
and (\ref{eq:MaximizationConstraint}), one gets the values \eqref{eqn: lambda1} and \eqref{eqn: lambda2} for $\lambda$.
}
\end{rem}

\medskip

With Remark \ref{rem: dicho} in mind, we are ready to bound the log-measure.

\begin{proof}[Proof of Lemma \ref{lem: high points}]
\noindent \textit{Upper bound on the log-measure.} 
For $0<\gamma<\gamma^\star$, consider $\mathcal{M}_T$ as in \eqref{eqn: M}.
We need to show that for $\e>0$
\begin{equation}
\label{eqn: to prove UB}
\lim_{T\to\infty} \P\left(\mathcal{M}_T > (\log T)^{\mathcal E_{\alpha,u}(\gamma)+\e}\right)=0\ .
\end{equation}
We first prove the easiest cases where \(u\ge 0\) and \(\gamma< \gamma_c\), as well as \(u\leq  0\). 
Let $\e>0$. And write $V=1-\alpha +(1+u)^2 \alpha$ for short. Observe that by Markov's inequality and Fubini's theorem
\begin{equation}
\begin{aligned}
\P(\mathcal M_T>(\log T)^{-\frac{\gamma^2}{V}+\e})
&\leq (\log T)^{\frac{\gamma^2}{V}-\e} \int_0^1 \P(\widetilde X_h>\gamma\log\log T)\ \rd h\\
& =(\log T)^{\frac{\gamma^2}{V}-\e} \P(\widetilde X_h>\gamma\log\log T)\ ,
\end{aligned}
\end{equation}
where we used the fact that the variables $\widetilde X_h$, $h\in[0,1]$, are identically distributed. 
Since  $\P(\widetilde X_h>\gamma\log\log T)\ll \exp(-\gamma^2\log\log T/V)$ by Equation \eqref{eqn: X tilde one-point}, the claim \eqref{eqn: to prove UB} follows.

The case \(u> 0\), \(\gamma> \gamma_c\) is more delicate as we need to control the values at scale $\alpha$.
For $\e'>0$ to be fixed later, note that the same argument as for equation \eqref{eqn: max interval} gives
\begin{equation}
\label{eqn: measure scale alpha}
\begin{aligned}
&\P\left({\rm Leb}\{h\in [0,1]: X_h(\alpha)>(\alpha+\e')\log\log T\}>0 \right)\\
&\hspace{2cm} \leq \P\left(\exists h\in [0,1]: X_h(\alpha)>(\alpha+\e')\log\log T \right)\to 0\ .
\end{aligned}
\end{equation}
The same hold by symmetry for $-X_h(\alpha)$. This implies
$$
\begin{aligned}
&\P\left(\mathcal{M}_T > (\log T)^{\mathcal E_{\alpha,u}(\gamma)+\e}\right)\\
&=\P\left({\rm Leb}\{h: \widetilde X_h>\gamma\log\log T,|X_h(\alpha)|\leq (\alpha+\e')\log\log T\}>(\log T)^{\mathcal E_{\alpha,u}(\gamma)+\e}\right)
+\oo(1)\ .
\end{aligned}
$$
It remains to prove that the first term is $\oo(1)$. 
As in the proof of Lemma \ref{lem: max}, we consider the partition of $[0,1]$ by intervals $J_j$, $j\leq (\log T)^{\alpha}$, and the sub-partition $I_{jk}$, $k\leq (\log T)^{1-\alpha}$.
We also divide the interval  $[-(\alpha+\e')\log\log T, (\alpha+\e')\log\log T]$ into intervals $[q,q+1]$. Then by Markov's inequality and the additivity of the Lebesgue measure
\begin{equation}
\label{eqn: leb decomp}
\begin{aligned}
&\P\left({\rm Leb}\{h: \widetilde X_h>\gamma\log\log T,|X_h(\alpha)|\leq (\alpha+\e')\log\log T\}>(\log T)^{\mathcal E_{\alpha,u}(\gamma)+\e}\right)\\ 
&\leq(\log T)^{-\mathcal E_{\alpha,u}(\gamma)-\e} \sum_{j,k}\sum_{|q|\leq (\alpha+\e')\log\log T} \E\left[{\rm Leb}\{h\in I_{jk}:  \widetilde X_h>\gamma\log\log T,X_h(\alpha)\in[q,q+1] \}\right]\\
& \leq (\log T)^{-\mathcal E_{\alpha,u}(\gamma)-\e}\sum_{j,k}\sum_{|q|\leq (\alpha+\e')\log\log T} (\log T)^{-1}  \P\left(X_h(\alpha,1)>\gamma\log\log T-(1+u)(q+1),X_h(\alpha)\geq q \}\right)\ .
\end{aligned}
\end{equation}
The last line follows from Fubini's theorem and the fact that ${\rm Leb}(I_{jk})=(\log T)^{-1}$. The probabilities can be bounded by the Gaussian bound \eqref{eqn: increment one point}
$$
\begin{aligned}
&\P\left(X_h(\alpha,1)>\gamma\log\log T-(1+u)(q+1),X_h(\alpha)\geq q \}\right)\\
&\ll \exp\left(\frac{-q^2}{\alpha \log\log T}\right)\cdot \exp\left(\frac{-\Big(\gamma\log\log T-(1+u)(q+1)\Big)^2}{(1-\alpha) \log\log T}\right)\ .
\end{aligned}
$$
It is easily checked that the expression is maximized at $q>(\alpha+\e')\log\log T$ for $\e'$.  
Moreover, at the optimal $q=(\alpha+\e')\log\log T$ in the considered range, the probability equals $(1+\oo(1))(\log T)^{\mathcal E_{\alpha,u}(\gamma)}$.
Using this observation to bound the probability for each $q$ in \eqref{eqn: leb decomp}, we get
$$
\begin{aligned}
&\P\left({\rm Leb}\{h: \widetilde X_h>\gamma\log\log T,|X_h(\alpha)|\leq (\alpha+\e')\log\log T\}>(\log T)^{\mathcal E_{\alpha,u}(\gamma)+\e}\right)\\ 
&\ll(\log T)^{-\e} \log\log T =\oo(1)\ .
\end{aligned}
$$
This finishes the proof of the upper bound.\\

\noindent \textit{Lower bound on the log-measure.} 
For $\e>0$, the goal is to show 
\begin{equation}\label{eq: lower bound goal}
\P\left(\mathcal{M}_T> (\log T)^{\mathcal E_{\alpha,u}(\gamma)-\e} \right)\to 1 \ \text{as $T\to\infty$.}
\end{equation}
This is done using the Paley-Zygmund inequality, which states that for a random variable \(\mathcal{M}\ge 0\) and \(0\le \eta_T\le
1\),
\begin{equation}
\P\left( \mathcal{M}>\eta _TE\left[ \mathcal{M}  \right]  \right)
\ge \left(1-\eta _T  \right)^2\frac{\E\left[ \mathcal{M}  \right]^2}{\E\left[ \mathcal{M}^2
\right]}.\label{eq:PaleyZygmund}
\end{equation}
We will have \(\eta_T\to 0\), so the main task will be to demonstrate
\begin{equation}
\E\left[ \mathcal{M}^2  \right]=\left( 1+\oo\left( 1 \right) \right)\E\left[
\mathcal{M}  \right]^2\ .
\label{eq:ShowPaleyZygmund}
\end{equation}
This cannot be achieved when \(\mathcal{M}=\mathcal{M}_T\)
because of the correlations in  \(X_h\). To overcome this problem, we define a modified version of \(\mathcal{M}_T\)
by \textit{coarse graining}  the field as described in \cite{kistler}.

For $K\in \N$ (that will depend eventually on $\e$), assume without loss of generality that
\(\{0,\frac{1}{K},\frac{2}{K},\dots,
\frac{K-1}{K},1\}\) is a partition of \([0,1]\) that is a refinement of \(\{0,\alpha,1\}\).
Consider $\lambda<\lambda^\star$ as defined in \eqref{eqn: lambda1} and \eqref{eqn: lambda2}, and $\delta>0$ (that will depend on $\e$).
 Define the  events for the \(K\)-level coarse increments:\begin{equation}\label{eq: corase increments}
\mathcal{J}_h(m)=\begin{cases}\left\{ (1+u)X_h\left(\frac{m-1 }{K}
,\frac{m }{K}
\right)\ge(1+\delta) 
\frac{\lambda\log \log T}{\alpha K}\right\} & \text{if }m=1,\dots ,\alpha K, \\
\left\{  X_h\left(\frac{m-1 }{K} ,\frac{m }{K}
\right)\ge(1+\delta) 
\frac{(\gamma-\lambda)\log \log T}{(1-\alpha )K}\right\}&\text{if }m=\alpha K+1,\dots,
K. \end{cases}
\end{equation} 
Moreover, define the sets
\begin{equation}\label{eq: modified exceedance}
\mathcal A=\{h: \mathcal{J}_h(m) \text{ occurs } \forall m=2,\dots,K\} \qquad
\mathcal B=\{h: (1+u)X_h\left(\frac{1}{K}\right)\leq-\frac{\delta}{2}\log\log T\} \ .
\end{equation}
 Note that if \(h\in \mathcal A \), by adding up the inequalities in  \(\mathcal{J}_h(m)\), we have for $K$ large enough,
 \begin{equation}
\label{eq: delta 2 lower bound}
\widetilde X_h-(1+u)X
_h\left(
\frac{1 }{K} \right)\ge(1+\delta)\left(\gamma  -  
\frac{\lambda}{\alpha K}\right)\log \log T\ge\left( \gamma+\frac{\delta}{2} \right)\log \log T\ .
\end{equation} 
Therefore this implies the inclusion
$$
\mathcal A \subset \left\{h\in [0,1]: \widetilde X_h \ge \gamma\log\log T  \right\}\cup \mathcal B\ ,
$$
so that $\mathcal{M}_T\ge \mathrm{Leb}(\mathcal A) -\mathrm{Leb}(\mathcal B)$.
Equation \eqref{eqn: increment one point} and Fubini's theorem shows that $\E[\mathrm{Leb}(\mathcal B)]\ll(\log T)^{-\frac{\delta^2K}{4(1+u)^2}}$.
For \(K\) large enough, Markov's inequality then implies
$$
\P\left(\mathrm{Leb}\left\{h\in [0,1]:(1+u)X_h\left(
\frac{1 }{K} \right)\le  -\frac{\delta}{2}\log
\log T\right\}\le (\log T)^{\mathcal{E}_{\alpha,u}(\gamma)-\varepsilon}  \right)\to 1.
$$
The proof of \eqref{eq: lower bound goal} is then reduced to show
\begin{equation}\label{eq: only need to show}
\P(\mathrm{Leb}(\mathcal A)>2(\log T)^{ \mathcal{E}_{\alpha,u}(\gamma)-\varepsilon})
=
\P(\mathrm{Leb}(\mathcal A)>\eta_T\E[\mathrm{Leb}(\mathcal A)])
\to 1\ ,
\end{equation}
where \(\eta_T \) is defined by $2(\log T)^{ \mathcal{E}_{\alpha,u}(\gamma)-\varepsilon}=\eta_T\E[\mathrm{Leb}(\mathcal A)]$.

Following \eqref{eq:PaleyZygmund}, we first show $\eta_T\to0$.
By (\ref{eq: modified exceedance}), Fubini's theorem,  and independence,
\begin{equation}
\label{eq: product of prob increments}
\E[\mathrm{Leb}(\mathcal A)]=\int_0^1\prod
_{m=2}^K\P\left( \mathcal{J}_h(m) \right) \rd h
=\prod_{m=2}^K\P\left( \mathcal{J}_h(m) \right),
\end{equation}
since the $X_h$'s are identically distributed.
By Proposition \ref{prop: berry-esseen}, 
\begin{equation}
\label{eq: probability of increment}\P (\mathcal{J}_h(m))
\gg
\begin{cases}(\log
T)^{- 
\frac{ (1+\delta)^2\lambda^2 }{\alpha^2 K(1+u)^2}    
  +\oo(1)} & \text{when }m=1,\dots ,\alpha K, \\
(\log T)^{-\frac{(1+\delta)^2  
(\gamma-\lambda)^2}{(1-\alpha )^2K} 
  +\oo(1)} & \text{when }m=\alpha K+1,\dots,
K. \\
\end{cases} 
\end{equation}Thus, by (\ref{eq: product of prob increments}) and (\ref{eq: probability of increment}), we have
\begin{equation}
\label{eq: expectation of modified}
\E [ \mathrm{Leb}(\mathcal A)]  
\gg(\log
T)^{ - \frac{ (1+\delta)^2\lambda^2 }{\alpha (1+u)^2}-\frac{(1+\delta)^2  
(\gamma-\lambda)^2}{(1-\alpha )}}(\log T)^{\frac{(1+\delta)^2\lambda^2 }{\alpha^2 K(1+u)^2}+\oo(1)}\ .
\end{equation} 
We can take $\lambda$ close enough to $\lambda^\star$, $\delta$ small enough, and $K$ large enough so that
$$
\E [ \mathrm{Leb}(\mathcal A)]  
\gg(\log
T)^{ - \frac{ {\lambda^\star}^2 }{\alpha (1+u)^2}-
\frac{  (\gamma-\lambda^\star)^2}{(1-\alpha )}-\frac{\e}{2}}=(\log T)^{  \mathcal E_{\alpha,u}(\gamma)-\frac{\e}{2}}\ ,
$$
where we replace the value of $\lambda^\star$ of \eqref{eqn: lambda1} and \eqref{eqn: lambda2}. This shows that \(\eta_T\to 0\).  
Observe that, we also have the reverse inequality 
\begin{equation}
\label{eqn: reverse}
\E [ \mathrm{Leb}(\mathcal A)] \ll(\log T)^{  \mathcal E_{\alpha,u}(\gamma)+\frac{\e}{2}}\ ,
\end{equation}
using \eqref{eqn: increment one point} instead of Proposition \ref{prop: berry-esseen}.

It remains to show (\ref{eq:ShowPaleyZygmund}). 
By independence of increments and Fubini's theorem, we have
\begin{equation}\label{eq: second moment}
\E[{\mathrm{Leb}(\mathcal A)}^2]= 
\int _0^1\int _0^1\prod _{m=2}^{  K} \P(\mathcal{J}_h(m)\cap \mathcal{\mathcal{J}}_{h'}(m))\ 
\rd h\rd h'\ .
\end{equation}
We split the integral into four integrals: I for $|h-h'|>(\log T)^{\frac{-1}{2K}}$, II for $(\log T)^{\frac{-1}{K}}\leq |h-h'|\leq(\log T)^{\frac{-1}{2K}}$,
III for  $(\log T)^{\frac{-r}{K}}<|h-h'|\leq(\log T)^{\frac{-(r-1)}{K}}$, $r=2,\dots K$, and IV for $|h-h'|\leq (\log T)^{-1}$.
We will show  that \(\mathrm I=\E\left[{\mathrm{Leb}(\mathcal A)}  \right]^2\left( 1+\oo(1) \right)\) and the others \(\oo(\E[{\mathrm{Leb}(\mathcal A)} ]^2\)).

\begin{itemize}
\item For $\mathrm{II}$, note that $\mathrm{Leb}^{\times 2}\{(h,h'):(\log T)^{\frac{-1}{K}}\leq |h-h'|\leq(\log T)^{\frac{-1}{2K}} \}\ll(\log T)^{\frac{-1}{2K}}$.
Moreover, by \eqref{eqn: increment two points} and Proposition \ref{prop: berry-esseen},  we have $\P (\mathcal{J}_h(m)\cap \mathcal{\mathcal{J}}_{h'}(m))\ll  \P (\mathcal{J}_h(m))^2$.
This implies $\mathrm{II}= \oo(\E[{\mathrm{Leb}(\mathcal A)} ]^2)$. 
\item For $\mathrm{IV}$, note that clearly $\P (\mathcal{J}_h(m)\cap \mathcal{\mathcal{J}}_{h'}(m))\leq \P (\mathcal{J}_h(m))$. Thus, $\mathrm{IV}\ll (\log T)^{-1}\E[\mathrm{Leb}(\mathcal A)]$. Using \eqref{eqn: reverse} and the fact that $1+ \mathcal{E}_{\alpha,u}(\gamma)>0$
for $\gamma<\gamma^\star$, one gets $\mathrm{IV}= \oo(\E[{\mathrm{Leb}(\mathcal A)} ]^2)$.

\item For $\mathrm{I}$, note that $\mathrm{Leb}^{\times 2}\{(h,h'): |h-h'|>(\log T)^{\frac{-1}{2K}} \}=1+\oo(1)$. 
Moreover, by Proposition \ref{prop: berry-esseen},  $\P (\mathcal{J}_h(m)\cap \mathcal{\mathcal{J}}_{h'}(m))=(1+\oo(1))  \P (\mathcal{J}_h(m))^2$.
This implies $\mathrm{I}= (1+\oo(1))\E[{\mathrm{Leb}(\mathcal A)} ]^2)$.

\item For $\mathrm{III}$, the integral is a sum over $r=2,\dots,K$ of integrals of pairs with $(\log T)^{\frac{-r}{K}}<|h-h'|\leq(\log T)^{\frac{-(r-1)}{K}}$.
The measure of this set is $\ll (\log T)^{\frac{-(r-1)}{K}}$.
For fix $r$, the integrand is
$$
\begin{aligned}
\prod_{m=2}^K \P(\mathcal{J}_h(m)\cap \mathcal{\mathcal{J}}_{h'}(m))
&\leq \prod_{m=2}^r \P(\mathcal{J}_h(m))\prod_{m=r+1}^K \P(\mathcal{J}_h(m)\cap \mathcal{\mathcal{J}}_{h'}(m))\\
&\ll  \prod_{m=2}^r \P(\mathcal{J}_h(m))\prod_{m=r+1}^K \P(\mathcal{J}_h(m))^2\ ,
\end{aligned}
$$
where the last line follows by \eqref{eqn: increment two points} and Proposition \ref{prop: berry-esseen}.
Putting all this together and factoring the square of the one-point probabilities, one gets
$$
\mathrm{III}
\ll\E[{\mathrm{Leb}(\mathcal A)} ]^2 \sum_{r=2}^K(\log T)^{\frac{-(r-1)}{K}} \prod_{m=2}^r \Big(\P(\mathcal{J}_h(m))\Big)^{-1}\ .
$$
We show $\prod_{m=2}^r \Big(\P(\mathcal{J}_h(m))\Big)^{-1}<(\log T)^{\frac{(r-1)}{K}}$ uniformly in $T$. 
This finishes the proof since the sum is then the tail of a convergent geometric series. In the case $u<0$, since $\lambda<\lambda^\star$, and $(1+\delta)\gamma<\gamma^\star$ for $\delta$ small, we have by \eqref{eq: probability of increment},
$$
\P(\mathcal{J}_h(m))^{-1}
\ll
\begin{cases}
(\log T)^{\frac{{\lambda^\star}^2}{\alpha^2K(1+u)^2}} \ &\text{ if $m\leq \alpha K$}\\
(\log T)^{\frac{(\gamma^\star-\lambda^\star)^2}{(1-\alpha)^2K}} \ &\text{ if $m=\alpha K+1,\dots,K$.}\\
\end{cases}
$$
By the definition of $\lambda^\star$ and $\gamma^\star=V^{1/2}$, this implies
$$
 \prod_{m=2}^r \Big(\P(\mathcal{J}_h(m))\Big)^{-1}
 \ll
\begin{cases}
(\log T)^{\frac{(1+u)^2}{V}\frac{r-1}{K}} \ &\text{ if $r\leq \alpha K$}\\
(\log T)^{\frac{\alpha(1+u)^2}{V}+\frac{1}{V} \frac{r-\alpha K}{K}} \ &\text{ if $r=\alpha K+1,\dots,K$.}\\
\end{cases}
 $$
 Since $u<0$, it is straightforward to check that the exponent is smaller than $\frac{r-1}{K}$ as claimed. 
 The case $u\geq 0$ is done similarly by splitting into two cases $\gamma_c\leq\gamma<\gamma^\star$ and $0<\gamma<\gamma_c$. 
 We omit the proof for conciseness.
 \end{itemize}
\end{proof}

We now have all the results to finish the proof of Proposition \ref{prop: free energy} using Laplace's method.
\begin{proof}[Proof of Proposition  \ref{prop: free energy}]
We first prove the limit in probability. The convergence in $L^1$, and in particular the convergence of the expectation, will be a consequence of Lemma \ref{lem: UI} below.
 For fixed $\e>0$ and $M\in \N$, consider 
$$
\gamma_j=\frac{j(1+\e)}{M}\gamma^\star \qquad 0\leq j\leq M\ ,
$$
and the event
\begin{multline}
A=\bigcap_{j=1}^M
\left\{(\log T)^{\mathcal E_{\alpha,u}(\gamma_j)-\e}\leq {\rm Leb}\{h: \widetilde X_h> \gamma_j \log \log T\} \leq (\log T)^{\mathcal E_{\alpha,u}(\gamma_j)+\e} \right\}\\
\bigcap \left\{ {\rm Leb}\{h: \widetilde X_h> \gamma_M\log \log T\}=0\right\}\ .
\end{multline}
By Lemma \ref{lem: max} and Lemma \ref{lem: high points}, we have that $\P(A^c)\to 0$ as $T\to\infty$. It remains to prove that the free energy is close to the claimed expression on the event $A$. 
On one hand, the following upper bound holds on $A$:
$$
\begin{aligned}
\int_0^1 \exp \beta \widetilde X_h \ \rd h&\leq \sum_{j=1}^M \int_0^1 \exp \beta \widetilde X_h\ \1_{\{(\log T)^{\gamma_{j-1}}<e^{\widetilde X_h}\leq(\log T)^{\gamma_{j}}\}} \ \rd h
+\int_0^1 \exp \beta \widetilde X_h\ \1_{\{e^{\widetilde X_h}<1\}} \ \rd h\\
& \leq \sum_{j=1}^M (\log T)^{\beta \gamma_{j}+\mathcal E_{\alpha,u}(\gamma_{j-1})+\e}+1\ .
\end{aligned}
$$
On the other hand, we have the lower bound
$$
\begin{aligned}
\int_0^1 \exp \beta \widetilde X_h \ \rd h&\geq \sum_{j=1}^M \int_0^1 \exp \beta \widetilde X_h\ \1_{\{(\log T)^{\gamma_{j-1}}<e^{\widetilde X_h}\leq(\log T)^{\gamma_{j}}\}} \ \rd h
& \geq \sum_{j=1}^M (\log T)^{\beta \gamma_{j-1}+\mathcal E_{\alpha,u}(\gamma_{j})-\e}\ .
\end{aligned}
$$
Altogether, this implies
$$
\max_{1\leq j\leq M}\left\{\beta \gamma_{j-1}+\mathcal E_{\alpha,u}(\gamma_{j})-\e\right\}  \leq \frac{\log \int_0^1 \exp \beta \widetilde X_h \ \rd h}{\log\log T}  \leq \max_{1\leq j\leq M}\left\{\beta \gamma_{j}+\mathcal E_{\alpha,u}(\gamma_{j-1})+\e\right\}+\oo(1)\ .
$$
In particular, by continuity of $\mathcal E_{\alpha,u}(\gamma)$, we can pick $M$ large enough depending on $\e$ and $T$ large enough so that
$$
\left|\frac{\log \int_0^1 \exp \beta \widetilde X_h \ \rd h}{\log\log T} -\max_{\gamma\in [0,\gamma^{\star}]}\left\{\beta \gamma+\mathcal E_{\alpha,u}(\gamma)\right\} \right|
\leq 2\e\ .
$$
As mentioned above, since $\P(A^c)\to 0$ as $T\to \infty$, this proves the convergence in probability
$$
\lim_{T\to\infty}\frac{\log \int_0^1 \exp \beta \widetilde X_h \ \rd h}{\log\log T}= \max_{\gamma\in [0,\gamma^{\star}]}\left\{\beta \gamma+\mathcal E_{\alpha,u}(\gamma)\right\}\ .
$$
It remains to check that the right side has the desired form. Let $V=(1+u)^2 \alpha+(1-\alpha)$. 
If $u<0$, the optimal $\gamma$ is $\beta V/2$ whenever $\beta V/2<\gamma^\star$, i.e., $\beta<2/V^{1/2}$. If $\beta\geq2/V^{1/2}$, then the optimal $\gamma$ is simply $\gamma^\star$. Therefore, we have
$$
 \max_{\gamma\in [0,\gamma^{\star}]}\left\{\beta \gamma+\mathcal E_{\alpha,u}(\gamma)\right\}
 =\begin{cases}
 \frac{\beta^2V}{4} \ \ &\text{if $\beta<2/V^{1/2}$}\\
 \beta V^{1/2}-1\ \ &\text{if $\beta\geq2/V^{1/2}$}\ .
 \end{cases}
 $$
 If $u\geq 0$, the optimal $\gamma$ is $\beta V/2$ if $\gamma <\gamma_c$, i.e., $\beta <2/(1+u)$. If $\gamma>\gamma_c$, then the optimal $\gamma$ is $(1+u)\alpha+\frac{\beta(1-\alpha)}{2}$ until it equals $\gamma^\star$. This happens at $\beta\geq 2$. Putting all this together, we obtain that
 $$
  \max_{\gamma\in [0,\gamma^{\star}]}\left\{\beta \gamma+\mathcal E_{\alpha,u}(\gamma)\right\}=
  \begin{cases}
   \frac{\beta^2\big((1+u)^2 \alpha+(1-\alpha)\big)}{4} \ \ &\text{if $\beta<\frac{2}{(1+u)}$}\\
    \beta (1+u)\alpha -\alpha + \frac{\beta^2(1-\alpha)}{4}\ \ &\text{if $\frac{2}{(1+u)}\leq \beta<2$}\\
 \beta \Big((1+u)\alpha+(1-\alpha)\Big)-1\ \ &\text{if $\beta\geq2$}\ .
  \end{cases}
  $$
  This corresponds to the expression in Proposition \ref{prop: free energy} expressed in terms of \eqref{eqn: f}.
\end{proof}

\begin{lem}
\label{lem: UI}
The sequence of random variables 
$$
\Big(\frac{1}{\log\log T}\log \int_{0}^1 \exp\big(\beta (X_h+uX_{h}(\alpha)\big) \rd h\Big)_{T>1}
$$
is uniformly integrable. In particular, the convergence in probability of the sequence is equivalent to the convergence in $L^1$.
\end{lem}
\begin{proof}
Write for short
$$
f_T
=(\log\log T)^{-1}\log \int_{0}^1 \exp \beta \widetilde X_h\ \rd h\ .
$$
We need to show that for any $\varepsilon>0$, there exists $C$ large enough so that uniformly in $T$,
$$
\E[|f_T| \1_{\{|f_T)|>C\}}]<\varepsilon \ .
$$
It is easy to check that
\begin{equation}
\label{eqn: UI}
\E[|f_T| \1_{\{|f_T|>C\}}]=\int_C^\infty \P(f_T>y) \ \rd y + C  \P(f_T>C) + \int_{-\infty}^{-C} \P(f_T<y) \ \rd y + C  \P(f_T<-C)\ .
\end{equation}
Therefore, it remains to get a good control on the right and left tail of $f_T$. For the right tail, observe that by Markov's inequality
$$
\begin{aligned}
\P(f_T>y) =\P\Big(\int\exp\beta \widetilde X_h\ \rd h>(\log T)^y\Big)
\leq (\log T)^{-y}\ \E\left[\int\exp\beta \widetilde X_h\ \rd h\right]\ .
\end{aligned}
$$
Using Proposition \ref{prop: MGF} and Fubini's theorem, we get
$$
\P(f_T>y) \ll (\log T)^{((1+u)^2\alpha+(1-\alpha))\frac{\beta^2}{4}-y}\ .
$$
This implies
$$
\int_C^\infty \P(f_T>y) \ \rd y + C  \P(f_T>C) 
\ll\frac{(\log T)^{((1+u)^2\alpha+(1-\alpha))\frac{\beta^2}{4}-C}}{\log\log T}+C(\log T)^{((1+u)^2\alpha+(1-\alpha))\frac{\beta^2}{4}-C}\ .
$$
It suffices to take $C>((1+u)^2\alpha+(1-\alpha))\frac{\beta^2}{4}$ for this to be uniformly small in $T$. The left tail is bounded the same way after noticing that by Markov's and Jensen's inequalities,
$$
\begin{aligned}
\P(f_T<-y) =\P\Big(\int \exp\beta \widetilde X_h\ \rd h<(\log T)^{-y}\Big)
&\leq (\log T)^{-y}\ \E\left[\left(\int \exp\beta \widetilde X_h\ \rd h\right)^{-1}\right]\\
&\leq  (\log T)^{-y}\ \E\left[\int \exp-\beta \widetilde X_h\ \rd h\right]\\
&  \ll (\log T)^{((1+u)^2\alpha+(1-\alpha))\frac{\beta^2}{4}-y}\ .
\end{aligned}
$$
These estimates imply that $\E[|f_T| \1_{\{|f_T)|>C\}}]$ can be made arbitrarily small in \eqref{eqn: UI} by taking $C$ larger than $((1+u)^2\alpha+(1-\alpha))\frac{\beta^2}{4}$.
\end{proof}

\bibliographystyle{plain}

\bibliography{bib_rzf_rsb}
\end{document}